\documentclass{article}

\usepackage[all]{xy}
\usepackage{amsmath,amssymb}
\usepackage{amsthm}
\usepackage{enumerate}

\theoremstyle{plain}
  \newtheorem{theorem}{Theorem}[section]
  \newtheorem{corollary}[theorem]{Corollary}
  \newtheorem{lemma}[theorem]{Lemma}
  
  \newtheorem{proposition}[theorem]{Proposition}

\theoremstyle{definition}
  \newtheorem{definition}[theorem]{Definition}
  
  \newtheorem{example}[theorem]{Example}
  
  \newtheorem{remark}[theorem]{Remark}
    
\newcommand{\Q}{\mathbb{Q}}

\newcommand{\op}{\mathrm{op}}

\newcommand{\ob}{\mathrm{ob}}
\newcommand{\id}{\mathrm{id}}

\newcommand{\ra}{\longrightarrow}

\newcommand{\DF}{\mathrm{DF}}
\newcommand{\F}{\mathcal{F}}
\newcommand{\I}{\mathcal{I}}
\newcommand{\Po}{\mathcal{P}}

\author{Kohei Tanaka}
\title {\textbf{Discrete Euler integration over functions on finite categories}}

\begin{document}

\renewcommand{\thesection}{\arabic{section}}

\maketitle




\begin{abstract}
This paper provides the theory of integration with respect to Euler characteristics of finite categories.
As an application, we use sensors to enumerate the targets lying on a poset.
This is a discrete analogue to Baryshnikov and Ghrist's work on integral theory using topological Euler characteristics.
\end{abstract}


\section{Introduction}

It has long been known that the Euler characteristic is an important homotopy invariant of a space.
By regarding it as a topological measure, we can derive the theory of the integral with respect to the Euler characteristic ({\em Euler integration}).
Baryshnikov and Ghrist developed the theory of Euler integration, and they applied it to sensor networks \cite{BG09}, \cite{BG10}. 
They established a way to use sensors to enumerate targets in a filed.
Let us review briefly a simple case that they considered.

Consider a situation in which there are a finite number of targets $T$ lying on a topological space $X$.
Assume that each point of $X$ has a sensor recording the nearby targets, 
and each target $t \in T$ has a contractible target support: 
\[
U_{t}=\{x \in X \mid \textrm{the sensor at } x \textrm{ detects } t\}.
\]
The sensors return the counting function $h : X \to \mathbb{N} \cup \{0\}$ given by the number of detectable sensors at each point:
\[
h(x) = \{t \in T \mid x \in U_{t}\}^{\sharp}.
\]
Then, we can enumerate the targets by integrating with respect to the Euler characteristic $\chi$ (Theorem 3.2 of \cite{BG09}):
\[
T^{\sharp} = \int_{X} h d \chi.
\]

Our goal in this paper is to show a discrete analogue of the above.
The Euler characteristic is defined not only for topological spaces, 
but also for finite combinatorial objects, such as posets \cite{Rot64}, groupoids \cite{BD01}, and categories \cite{Lei08}.
We will focus on the Euler characteristics of finite categories, which is the most general case 
(note that a poset can be thought of as an acyclic category with at most one morphism between any pair of objects).  
By using this instead of topological Euler characteristic, we can derive a discrete version of Euler integration.
We treat certain rational-valued functions on objects of a category as integrable functions, 
since Euler characteristics of finite categories take rational values.

As an application of discrete Euler integration, we consider the counting problem in a network flowing in only one direction (such as transmission of electricity, streams of water or rivers, and acyclic traffic).
We propose a way to use sensors to enumerate the targets lying on a one-way network.

The remainder of this paper is organized as follows. 
Section 2 prepares some necessary definitions and notations for Euler characteristics of finite categories; we use the system created by Leinster \cite{Lei08}. 
In Section 3, we introduce the class of integrable ({\em definable}) functions on a finite category.
Following this, we define the Euler integration of definable functions, and investigate its properties.
Section 4 presents an application with sensor networks.

\section{Euler characteristics of finite categories}

The Euler characteristic of a finite category was introduced by Leinster \cite{Lei08}; 
it is a generalization of the concept of M\"{o}bius inversion \cite{Rot64} for posets.

\begin{definition}
Suppose that $C$ is a finite category consisting of a finite number of objects and morphisms.
We denote the set of objects of $C$ by $\ob(C)$, and the set of morphisms from $x$ to $y$ by $C(x,y)$.
\begin{enumerate}
\item The {\em similarity matrix} of $C$ is the function 
$\zeta : \mathrm{ob}(C) \times \mathrm{ob}(C) \to \mathbb{Q}$, given by the cardinality of each set of morphisms: $\zeta(a,b)=C(a,b)^{\sharp}$.
\item Let $u : \ob(C) \to \mathbb{Q}$ denote the column vector with $u(a)=1$, for any object $a$ of $C$.
A {\em weighting} on $C$ is a column vector $w:\mathrm{ob}(C)\to \mathbb{Q}$ such that $\zeta w=u$, and 
dually, a {\em coweighting} on $C$ is a row vector $v:\mathrm{ob}(C)\to \mathbb{Q}$ such that $v \zeta = u^{\ast}$, where $u^{\ast}$ is the transposition of the matrix $u$.
\end{enumerate}
\end{definition}

Note that we have
\[
\sum_{i \in \mathrm{ob}(C)}w(i) = u^{\ast} w = v \zeta w= v u=\sum_{j \in \mathrm{ob}(C)} v(j),
\]
if both a weighting and a coweighting exist. 
Moreover, 
\[
\sum_{i \in  \mathrm{ob}(C)} w(i)= u^{\ast} w = v \zeta w = v \zeta w' = u^{\ast} w' = \sum_{i \in \mathrm{ob}(C)} w'(i),
\]
for two (co)weightings $w$ and $w'$ on $C$. This guarantees the following definition of the Euler characteristic.

\begin{definition}[\cite{Lei08}]\label{chara}
Let $C$ be a finite category. 
We say that $C$ \textit{has Euler characteristic} if it has 
both a weighting $w$ and a coweighting $v$ on $C$. 
Then, the \textit{Euler characteristic} of $C$ is defined by
\[
\chi(C) = \sum_{i \in \mathrm{ob}(C)}w(i) = \sum_{j \in \mathrm{ob}(C)}v(j).
\]
\end{definition}

\begin{proposition}[Example 2.3 (d) in \cite{Lei08}]\label{contractible}
If $C$ has Euler characteristic and either an initial or a terminal object, then $\chi(C)=1$.
\end{proposition}

For a small category $C$, the classifying space $BC$ is defined as the geometric realization of the nerve of $C$.
When $C$ never has a nontrivial circuit of morphisms ({\em acyclic}), 
the following relation holds between the topological Euler characteristic and the combinatorial one defined in Definition \ref{chara}.

\begin{proposition}[Proposition 2.11 in \cite{Lei08}]\label{acyclic}
Every finite acyclic category $C$ has Euler characteristic, and $\chi(C)=\chi(BC)$.
\end{proposition}

For topological Euler characteristics, we have the following well-known inclusion-exclusion formula:
\[
\chi(A \cup B)=\chi(A) + \chi(B)-\chi(A \cap B),
\]
for suitable subspaces $A$ and $B$ of a space. 
Unfortunately, in the case of Euler characteristics of categories, it does not hold in general.
In \cite{Tan}, the author introduced two classes of categories satisfying the inclusion-exclusion formula stated above.

\begin{definition}
Let $C$ be a small category. A {\em filter} $D$ is a full subcategory of $C$ such that an object $y$ of $C$ belongs to $D$ whenever $C(x,y) \not= \emptyset$ for some object $x$ of $D$.
Dually, an {\em ideal} $D$ of $C$ is a full subcategory such that an object $y$ of $C$ belongs to $D$ whenever $C(y,x) \not= \emptyset$ for some object $x$ of $D$. in other words, a full subcategory $D$ of $C$ is an ideal if and only if the opposite category $D^{\op}$ is a filter of $C^{\op}$.
\end{definition}

These are generalizations for posets of filters and ideals \cite{Sta12}, \cite{Zap98}.
Let $D$ be a full subcategory of a small category $C$. The category of complements $C \backslash D$ is defined as the full subcategory of $C$ whose set of objects is $\ob(C) \backslash \ob(D)$.
If $D$ is a filter (an ideal) of $C$, then $C \backslash D$ is an ideal (a filter) of $C$.
In other words, a filter or an ideal $D$ determines a functor $C \to \mathbb{I}$, where $\mathbb{I}$ is the poset formed of $0<1$.
Hence, we have three bijective sets: the set $\F(C)$ of filters of $C$, the set $\I(C)$ of ideals of $C$, and the set $C^{\mathbb{I}}$ of functors from $C$ to $\mathbb{I}$.

The following theorem is shown in \cite{Tan}.

\begin{theorem}[Corollary 3.4 of \cite{Tan}]\label{inclusion-exclusion}
Let both $A$ and $B$ be either filters or ideals of a finite category $C$. If each of $A$, $B$, $A \cap B$, and $A \cup B$ has an Euler characteristic, then we have
\[
\chi(A \cup B)=\chi(A)+\chi(B)-\chi(A \cap B).
\]
\end{theorem}

\section{Discrete Euler integration over functions on categories}

Throughout this paper, we will occasionally identify a full subcategory $B$ of a category $C$ as the underlying set $\ob(B)$ of objects.
Moreover, for two categories $C$ and $D$, a map on objects $\ob(C) \to \ob(D)$ will be simply denoted by $C \to D$.

\begin{definition}
Let $C$ be a small category.
\begin{enumerate}
\item Two objects $x$ and $y$ of $C$ are {\em reflexible} if $C(x,y) \neq \emptyset$ and $C(y,x) \neq \emptyset$. In this case, let us denote $x \sim y$.
We then have an equivalence relation on the set of objects of $C$. 
The quotient set $\Po(C)$ is equipped with a partial order that is defined by $[x] \leq [y]$ if $C(x,y) \neq \emptyset$.
Here, this order does not depend on the choice of representing objects.
\item Let $\vee x$ denote the {\em prime filter} generated from an object $x$ of $C$, i.e., it consists of ending objects of morphisms starting at $x$:
\[
\vee x =\{ y \in \ob(C) \mid C(x,y) \neq \emptyset\}.
\]
Dually, let $\wedge x$ denote the {\em prime ideal} generated from $x$ and consisting of starting objects of morphisms ending at $x$. 
\item An object $x$ of $C$ is called {\em maximal} if it satisfies either $C(x,y)=C(y,x)=\emptyset$ or $C(y,x) \neq \emptyset$ for any object $y$.
Dually, a {\em minimal} object $x$ of $C$ satisfies either $C(x,y)=C(y,x)=\emptyset$ or $C(x,y) \neq \emptyset$ for any object $y$.
\end{enumerate}
\end{definition}

Baryshnikov and Ghrist introduced a class of integrable ({\em definable, constructible}) functions with respect to the topological Euler characteristic \cite{BG09}, \cite{BG10}.
Below, we give the definition for definable maps in our discrete setting.

\begin{definition}
A map $f : C \to D$ on objects is called {\em definable} if it preserves the reflexible relation.
That is, $f(x) \sim f(y)$ in $D$ for any reflexible pair $x \sim y$ in $C$.
In particular, a definable map $C \to \mathbb{Q}$ is called a {\em definable function} on $C$. 
Here, we regard the set of rational numbers $\mathbb{Q}$ as a totally ordered set in the canonical order. Therefore, a definable function sends a reflexible pair to the same value.
Let $\DF(C)$ denote the $\mathbb{Q}$-vector space of definable functions on $C$. 
For a poset $P$, every $\mathbb{Q}$-valued function on $P$ is definable, and hence, 
the vector space $\DF(P)$ consists of $\mathbb{Q}$-valued functions on $P$.
\end{definition}

\begin{remark}\label{remark}
The above definitions related to a category $C$ can be described in terms of the poset $\Po(C)$ and the canonical functor $\pi : C \to \Po(C)$.
\begin{itemize}
\item A full subcategory $D$ of $C$ is a filter (an ideal) of $C$ if and only if $\pi(D)$ is a filter  (an ideal) of $\Po(C)$.
\item An object $x$ of $C$ is maximal (minimal) if and only if $\pi(x)$ is maximal (minimal) in $\Po(C)$.
\item A map $f : C \to D$ on objects of two categories $C$ and $D$ is definable if and only if it induces a map $\tilde{f} : \Po(C) \to \Po(D)$ (it is not necessary that the order be preserved) that makes the following diagram commute:
\[
\xymatrix{
C \ar[r]^{f} \ar[d]_{\pi} & D \ar[d]^{\pi} \\
\Po(C) \ar[r]_{\tilde{f}} & \Po(D).
}
\]
\end{itemize}
\end{remark}

\begin{definition}
Let $B$ be a full subcategory of a category $C$, and let $f : C \to \mathbb{Q}$ be a function on objects.
The {\em clipping} of $f$ on $B$ is the function $f_{B} : C \to \mathbb{Q}$ defined by $f_{B}(x)=f(x)$ if $x \in \ob(B)$, and $f_{B}(x)=0$ otherwise.
This should not be confused with the restriction $f_{|B}$, whose domain is $B$.
The clipping $\delta_{B}$ of the constant function $\delta : C \to \mathbb{Q}$ onto $1 \in \mathbb{Q}$ is called the {\em incidence function} on $B$. 
The incidence function on $B$ is definable if $B$ is either a filter or an ideal of $C$.
\end{definition}

In Definition 2.3 of \cite{BG09}, a definable (constructible) function on a finite simplicial complex 
is defined to be a linear form of the incidence functions on simplices.

\begin{lemma}
A rational-valued function $f$ on a finite poset $P$ can be described using the following finite linear forms:
\[
f = \sum_{i=1}^{n} a_{i} \delta_{A_{i}}=\sum_{j=1}^{m}b_{j}\delta_{B_{j}},
\]
where $n,m \geq 1$; $a_{i},b_{j} \in \mathbb{Q}$; $A_{i} \in \F(P)$; and $B_{i} \in \I(P)$.
\end{lemma}
\begin{proof}
We use induction on the cardinality of $P$.
When $P$ consists of a single element $x$, it is obvious that $f(x)= f(x) \delta_{\vee x}=f(x) \delta_{\wedge x}$.  
Assume that any function on $P$ has a linear form of incidence functions on filters, in the case of $P^{\sharp} \leq n-1$.
When $P^{\sharp} = n$, take a maximal element $x \in P$. 
The inductive assumption decomposes the restriction $f_{|P-\{x\}}$ as 
\[
f_{|P-\{x\}} = \sum_{i=1}^{n}a_{i}\delta_{A_{i}},
\]
for $a_{i} \in \mathbb{Q}$, $A_{i} \in \F(P-\{x\})$. 
The union $A_{i} \cup \{x\}$ is a filter of $P$ for each $i$.
The function $f$ is described as
\[
f= \sum_{i=1}^{n} a_{i} \delta_{A_{i} \cup \{x\}} + \left( f(x)-\sum_{i=1}^{n}a_{i} \right) \delta_{\vee x}.
\]
On the other hand, each incidence function $\delta_{Q}$ on a filter $Q$ is described as the linear form 
$\delta_{Q} = \delta_{P}-\delta_{P \backslash Q}$ of incidence functions on ideals, and this completes the proof.
\end{proof}

\begin{corollary}
A definable function $f$ on a finite category $C$ can be described using the following finite linear forms:
\[
f = \sum_{i=1}^{n} a_{i} \delta_{A_{i}}=\sum_{j=1}^{m}b_{j}\delta_{B_{j}},
\]
where $n,m \geq 1$; $a_{i},b_{j} \in \mathbb{Q}$; $A_{i} \in \F(C)$; and $B_{i} \in \I(C)$.
\end{corollary}
\begin{proof}
The induced function $\pi(f)$ on the poset $\Po(C)$ has the desired liner forms:
\[
\pi(f) = \sum_{i=1}^{n} a_{i} \delta_{A_{i}}=\sum_{j=1}^{m}b_{j}\delta_{B_{j}},
\]
where $n,m \geq 1$; $a_{i},b_{j} \in \mathbb{Q}$; $A_{i} \in \F(\Po(C))$; and $B_{j} \in \I(\Po(C))$.
For the canonical projection $\pi : C \to \Po(C)$, the inverse images $\pi^{-1}(A_{i})$ are filters and $\pi^{-1}(B_{i})$ are ideals of $C$.
We can describe the function $f$ as
\[
f = \sum_{i=1}^{n} a_{i} \delta_{\pi^{-1}(A_{i})}=\sum_{j=1}^{m}b_{j}\delta_{\pi^{-1}(B_{j})}.
\]
\end{proof}

For a category $C$, the corollary above shows the following equalities:
\[
\mathrm{DF}(C)=\left\{ \sum_{i} a_{i}\delta_{A_{i}} \mid a_{i} \in \mathbb{Q}, A_{i} \in \F(C)\right\} 
= \left\{ \sum_{j} b_{j}\delta_{B_{j}} \mid b_{j} \in \mathbb{Q}, B_{j} \in \I(C)\right\}
\]

A filter $A$ of a finite poset $P$ can be written as the union of prime filters $\vee x$.
Hence, the vector space $\DF(P)$ of rational-valued functions on $P$ has two bases consisting of $\delta_{\vee x}$ and $\delta_{\wedge x}$ for each $x \in P$. 

In the case of a finite category $C$, we choose and fix objects $x_{1},\cdots,x_{k}$ such that $\Po(C)=\{[x_{1}],\cdots,[x_{k}]\}$.
The vector space $\DF(C)$ of definable functions on $C$ has two bases consisting of $\delta_{\vee x_{i}}$ and $\delta_{\wedge x_{i}}$ for $i=1,\cdots,k$.

\begin{definition}
A finite category $C$ is called {\em measurable} if each filter and ideal of $C$ has Euler characteristic. 
\end{definition}

For example, finite posets, acyclic categories, groups, and groupoids are measurable, since any full subcategory has Euler characteristic.

\begin{definition}
For a measurable category $C$, the {\em Euler integration on filters} is the linear map
\[
\int_{C}^{\F}(-)d \chi : \DF(C) \ra \Q,
\]
which sends $\delta_{\vee x}$ to $\chi(\vee x)$. Dually, the {\em Euler integration on ideals} is the linear map
\[
\int_{C}^{\I}(-)d \chi : \DF(C) \ra \Q,
\]
that sends $\delta_{\wedge x}$ to $\chi(\wedge x)$. 
\end{definition}

We should note that, in general, $\chi(D) \neq \chi(C)-\chi(C \backslash D)$ for a filter (an ideal) $D$ of $C$. 
This implies that
\[
\int_{C}^{\F} f d \chi \neq \int_{C}^{\I} f d \chi,
\]
for $f=\delta_{D}=\delta_{C}-\delta_{C \backslash D} \in \DF(C)$. 

\begin{remark}
We now have two kinds of integration via Euler characteristics of categories.
For a measurable category $C$, the duality between the filters and ideals implies that $C_{\F}=C^{\op}_{\I}$ and
\[
\int_{C}^{\F} f d \chi = \int_{C^{\op}}^{\I} f d \chi,
\]
for any $f \in \mathrm{DF}(C)=\mathrm{DF}(C^{\op})$. 
These two integrals are not equal, but have duality with the opposite category.
Henceforward, we shall only consider the case of Euler integration on filters.
We call it {\em Euler integration} for short, and $\int^{\F}_{C}$ will be simply denoted by $\int_{C}$. 
Euler integration on ideals satisfies the dual properties of that described below.
\end{remark}

\begin{lemma}
Euler integration does not depend on the linear representation of the definable functions with respect to the incidence functions on the filters.
That is, if a definable function $f$ can be described as $\sum_{i=1}^{n} a_{i}\delta_{A_{i}}$ for some $a_{i} \in \mathbb{Q}$ and $A_{i} \in \F(C)$, then we have 
\[
\int_{C} f d \chi = \sum_{i=1}^{n}a_{i}\chi(A_{i}).
\]
\end{lemma}
\begin{proof}
It suffices to show that
\begin{equation}
\begin{split}\label{independent}
\int_{C} f_{B} d \chi = \sum_{i=1}^{n}a_{i}\chi(A_{i} \cap B)
\end{split}
\end{equation}
for any filter $B$ of $C$. We use induction on the cardinality of the set of objects of $B$. 
When $B$ is a nonempty minimal filter, it is the prime filter $\vee b$ generated from some maximal object $b$. Any pair of objects in $B$ are reflexible, and the clipping $f_{B}=f(b)\delta_{\vee b}$ is constant on $B$.
Then, the following equality holds:
\[
\int_{C} f_{B} d \chi = f(b) \chi(\vee b) = \sum_{i;b \in A_{i}}a_{i}\chi(\vee b)=\sum_{i=1}^{n} a_{i} \chi(A_{i} \cap (\vee b)),
\]
since $A_{i} \cap (\vee b)=\vee b$ if $b \in A_{i}$, and $A_{i} \cap (\vee b)=\emptyset$ otherwise.

Next, assume that we have the equality stated in (\ref{independent}) for the clipping on $B$ of any definable function, when $B^{\sharp} \leq n-1$. 
When $B^{\sharp}=n$, choose a minimal object $x$ of $B$.
If the filter $A_{i}$ contains the object $x$, then $A_{i} \cap B=\vee x$.
The clipping function $f_{B}$ can be represented as
\[
f_{B}= f(x)\delta_{\vee x} +f_{B\backslash (\wedge x)}-f(x)\delta_{(\vee x) \cap (B\backslash (\wedge x))}.
\]
The Euler integration of the first term is obtained by the definition, and the second and third terms are obtained by the inductive assumption.
We can write the above terms as follows:
\[
f(x)\delta_{\vee x} = \sum_{i;x \in A_{i}}a_{i}\delta_{\vee x},
\]
\[
f_{B \backslash (\wedge x)} = \sum_{i ; x \not \in A_{i}} a_{i} \delta_{A_{i} \cap B}+\sum_{i; x \in A_{i}} a_{i} \delta_{A_{i} \backslash (\wedge x)},
\]
and
\[ 
f(x)\delta_{(\vee x) \cap (B\backslash (\wedge x))}=  \sum_{i;x \in A_{i}}a_{i}\delta_{(\vee x) \cap (B\backslash (\wedge x))} = \sum_{i;x \in A_{i}}a_{i}\delta_{A_{i} \backslash (\wedge x)}.
\]
By applying the Euler integration to $f_{B}$, we obtain the following equality:
\begin{equation}
\begin{split}
\int_{C} f_{B} d\chi &= \sum_{i; x \in A_{i}} a_{i}\chi(\vee x) + \sum_{i; x \not \in A_{i}} a_{i} \chi(A_{i} \cap B) \\ \notag
&=\sum_{i; x \in A_{i}} a_{i}\chi(A_{i} \cap B)+ \sum_{i; x \not \in A_{i}} a_{i} \chi(A_{i} \cap B) \\
&= \sum_{i=1}^{n} a_{i} \chi(A_{i} \cap B).
\end{split}
\end{equation}
The result corresponds to the case in which $B=C$.
\end{proof}

\begin{proposition}\label{value}
Let $f$ be a definable function on a measurable category $C$.
If $C$ has a terminal object $x$, then we have 
\[
\int_{C} f d \chi = f(x).
\]
\end{proposition}
\begin{proof}
It is sufficient to prove that 
\begin{equation}
\begin{split}\label{terminal}
\int_{C} f_{B} d \chi = f(x)
\end{split}
\end{equation}
for any filter $B$ of $C$. Note that every filter of a category includes terminal objects, if they exist. 
We use induction on the cardinality of $B$ to prove this. 
When $B$ is a nonempty minimal filter, it is the prime filter $\vee x$ generated from the terminal object $x$. For a definable function $f$ on $C$, we have
\[
\int_{C} f_{B} d \chi = \int_{C} f(x) \delta_{\vee x} = f(x) \chi(\vee x) = f(x),
\]
by Proposition \ref{contractible}.
We assume that our desired equality, stated in (\ref{terminal}), holds for any definable function if $B^{\sharp} \leq n-1$. 
When $B^{\sharp} =n$, we take a minimal object $b$ of $B$.
For a definable function $f$ on $C$, the clipping function $f_{B}$ can be represented as
\[
f_{B} = f_{B \backslash (\wedge b)} + f(b)\delta_{\vee b} - f(b) \delta_{(B \backslash (\wedge b)) \cap (\vee b)}.
\]
Here, each of $B \backslash (\wedge b)$, $\vee b$, and $(B \backslash (\wedge b)) \cap (\vee b)$ is a filter of $C$.
The inductive assumption leads to the following equality:
\[
\int_{C} f_{B} d \chi = f(x)+f(b)-f(b)=f(x).
\]
The result corresponds to the case in which $B=C$.
\end{proof}

Let $f$ be a definable function on a measurable category $C$, and let $B$ be a measurable full subcategory of $C$. For simplicity, the Euler integration $\int_{B} \left(f_{|B}\right) d \chi$ of the restriction $f_{|B}$ is denoted by $\int_{B} f d\chi$.

\begin{theorem}
Let $f$ be a definable function on a measurable category $C$, 
and let the category $C$ be the union $A \cup B$ of two full subcategories $A$ and $B$ of $C$.
If both $A$ and $B$ are either filters or ideals, then we have
\[
\int_{C} f d\chi = \int_{A} f d\chi + \int_{B} f d\chi -\int_{A \cap B} f d \chi.
\]
\end{theorem}
\begin{proof}
We will focus on the case in which both $A$ and $B$ are ideals.
It suffices to show the case of $f=\delta_{D}$ for a filter $D$ of $C$. 
Each $D \cap X$ is an ideal of $D$ and a filter of $X$, for $X=A$, $B$, and $A \cap B$.
Theorem \ref{inclusion-exclusion} induces the following equality:
\begin{equation}
\begin{split}
\int_{C} \delta_{D} d\chi &= \chi(D) \\ \notag
&= \chi(D \cap A) + \chi(D \cap B) - \chi(D \cap A \cap B) \\
&=\int_{A} \delta_{D} d\chi + \int_{B} \delta_{D} d\chi -\int_{A \cap B} \delta_{D} d \chi.
\end{split}
\end{equation}
The case in which $A$ and $B$ are filters can be shown similarly.
\end{proof}

\begin{definition}
Let $C$ and $D$ be categories. A map $F : C \to D$ is called {\em measurable} when the inverse image $F^{-1}(-)$ preserves filters and ideals.
\end{definition}

For example, the underlying map on objects of a functor is measurable.

\begin{lemma}
A measurable map $F : C \to D$ is definable.
\end{lemma}
\begin{proof}
Suppose that two objects $a$ and $b$ of $C$ are reflexible.
The inverse image $F^{-1}(\vee F(b))$ of the filter generated from $F(b)$ in $D$ is a filter in $C$.
The object $b$ belongs to this filter and so does $a$. This implies that $F(a)$ belongs to $\vee F(b)$.
Dually, the filter $F^{-1}(\vee F(a))$ contains the object $b$, and $F(b)$ belongs to $\vee F(a)$.
It follows that $F(a)$ and $F(b)$ are reflexible.
\end{proof}

At the rest of this section, we will examine the notion of {\em pushforwards} and Fubini theorem, introduced in Definition 2.7 and Theorem 2.8 of \cite{BG09}.
They used the inverse image of each point of $Y$ as the integral range, 
in order to define the pushforward of a definable map $X \to Y$. 
However, we use the inverse image of the ideal generated by each point as integral range.

\begin{definition}
Let $F : C \to D$ be a measurable map between measurable categories $C$ and $D$.
The {\em pushforward} of $F$ is the homomorphism $F_{\ast} : \DF(C) \to \DF(D)$ defined by
\[
F_{\ast}f(d) = \int_{F^{-1}(\wedge d)} f d \chi
\]
for $f \in \DF(C)$ and $d \in \ob(D)$.
\end{definition}

We need to verify that the pushforward described above is well-defined.

\begin{lemma}
For a measurable map $F: C \to D$ and a definable function $f \in \DF(C)$, 
the pushforward $F_{\ast}f$ is definable.
\end{lemma}
\begin{proof}
Suppose that two objects $a$ and $b$ of $D$ are reflexible.
We then have $\wedge a = \wedge b$, and
\[
F_{\ast}f(a) = \int_{F^{-1}(\wedge a)} f d \chi= \int_{F^{-1}(\wedge b)} f d \chi=F_{\ast}f(b).
\]
\end{proof}

\begin{proposition}\label{functorial}
Let $F : C \to D$ and $G : D \to E$ be two measurable map for measurable categories $C$, $D$, and $E$.
If $D$ is a poset, the pushforward is compatible with the composition, i.e.,
\[
(G \circ F)_{\ast} = G_{\ast} \circ F_{\ast}.
\]
\end{proposition}
\begin{proof}
The pushforward $(G \circ F)_{\ast} : \DF(C) \to \DF(E)$ is given by
\[
(G \circ F)_{\ast}(\delta_{B})(e) = \int_{(G \circ F)^{-1}(\wedge e)} \delta_{B} d \chi = \chi(F^{-1}(G^{-1}(\wedge e)) \cap B),
\]
for $B \in \F(C)$ and $e \in \ob(E)$.
We need to prove that $(G \circ F)_{\ast} (\delta_{B}) = G_{\ast}(F_{\ast}(\delta_{B}))$ for any $B \in \F(C)$.
It suffices to show that 
\begin{equation}
\label{(2)}
\begin{split}
\int_{A} F_{\ast}\delta_{B} d \chi= \chi(F^{-1}(A) \cap B)
\end{split}
\end{equation}
for any $e \in \ob(E)$ and ideal $A \subset G^{-1}(\wedge e)$. We again use induction on the cardinality of the ideal $A$.
When $A$ consists of a single (minimal) element $d$, 
\[
\int_{\{d\}} F_{\ast}\delta_{B} d \chi = F_{\ast}\delta_{B}(d) = \int_{F^{-1}(d)} \delta_{B} d \chi = \chi(F^{-1}(d) \cap B).
\]
Assume that equation (\ref{(2)}) holds in the case of $A^{\sharp} \leqq n-1$. 
When $A^{\sharp}=n$, we take a maximal element $a \in A$. By Proposition \ref{value}, Theorem \ref{inclusion-exclusion}, and the inductive assumption, we have the following equality:
\begin{equation}
\begin{split}
&\int_{A} F_{\ast}\delta_{B} d \chi = \int_{A \backslash (\vee a)}  F_{\ast}\delta_{B} d \chi+ \int_{\wedge a}  F_{\ast}\delta_{B} d \chi 
- \int_{(A \backslash (\vee a)) \cap(\wedge a)}  F_{\ast}\delta_{B} d \chi \\ \notag
=& \chi(F^{-1}(A \backslash (\vee a)) \cap B) + F_{\ast}\delta_{B}(a) - \chi(F^{-1}((A \backslash (\vee a)) \cap (\wedge a)) \cap B)\\
=& \chi(F^{-1}(A \backslash (\vee a)) \cap B) + \chi( F^{-1}(\wedge a) \cap B) - \chi(F^{-1}((A \backslash (\vee a)) \cap (\wedge a)) \cap B)\\
=& \chi(F^{-1}(A) \cap B).
\end{split}
\end{equation}
When $A=G^{-1}(\wedge e)$, we obtain the desired result:
\[
G_{\ast}(F_{\ast}(\delta_{B}))(e)=\int_{G^{-1}(\wedge e)} F_{\ast}(\delta_{B}) = \chi(F^{-1}(G^{-1}(\wedge e)) \cap B)= (G \circ F)_{\ast}(\delta_{B})(e).
\]
\end{proof}

In the proposition above, assuming $D$ to be a poset is essential to prove.
For example, if $D$ is a category consisting of two objects $a$ and $b$, and parallel two morphisms from $a$ to $b$. The classifying space of $D$ is homotopy equivalent to a circle $S^{1}$, and the Euler characteristic $\chi(D)=0$.
We have $(1_{D})_{\ast}(\delta_{D})(b)= \chi(D)=0$, however, 
\[
(1_{D})_{\ast} ((1_{D})_{\ast} (\delta_{D}))(b)=\int_{D} (\delta_{\vee a} +(-1) \delta_{\vee b}) d\chi = \chi(D)-\chi(\{b\})=-1. 
\]

\begin{corollary}
The pushforward yields a functor from the category of posets to the category of $\mathbb{Q}$-vector spaces, and it is given by $P \mapsto \DF(P)$ and $F \mapsto F_{\ast}$.
\end{corollary}
\begin{proof}
For the identity map $\id_{P} : P \to P$, the pushforward $F_{\ast}(\id_{P})=\id_{\DF(P)}$, since
\[
(\id_{P})_{\ast}(f)(a) = \int_{\wedge a} f d \chi =f(a)
\]
for any $f \in \DF(P)$ and $a \in P$, by Proposition \ref{value}.
Proposition \ref{functorial} completes the proof.
\end{proof}

\begin{theorem}\label{Fubini}
Let $F : C \to D$ be a measurable map from a measurable category $C$ to a finite poset $D$.
For any definable function $f$ on $C$, the Euler integration over $f$ coincides with the Euler integration over the pushforward of $f$:
\[
\int_{C} f d \chi = \int_{D} F_{\ast}f d \chi.
\]
\end{theorem}
\begin{proof}
Let $\mathbf{pt}$ denote the terminal category consisting of a single object and the identity morphism.
Note that the pushforward of the unique map $C \to \mathbf{pt}$ coincides with the Euler integration
\[
\int_{C} (-) d \chi : \DF(C) \to \DF(\mathbf{pt}) = \Q.
\]
By applying Proposition \ref{functorial} to the composition $C \stackrel{F}{\to} D \to \mathbf{pt}$, we obtain the desired formula.
\end{proof}

The canonical functor $\pi : C \to \Po(C)$ is a measurable map for a measurable category $C$.
Theorem \ref{Fubini} implies that 
\[
\int_{C} f d \chi = \int_{\Po(C)} \pi_{\ast}f d \chi.
\] 
Hence, the Euler integration of a definable function $f$ on a measurable category $C$ can be calculated from the function $\pi_{\ast}f$ on the poset $\Po(C)$.
Note that, in general, the pushforward $\pi_{\ast}f$ does not coincide with the induced map $\tilde{f}$ introduced in Remark \ref{remark}.

\section{Application of discrete Euler integration to sensor network theory}

This section presents an application of Euler integration over a function on a poset.
It is based on the work of Baryshnikov and Ghrist on using topological Euler integration to enumerate targets in a network of sensors in a field.
In \cite{BG09}, they proved that the cardinality of the targets lying on a field can be obtained from the topological Euler integral of the counting function.

We consider a discrete analogue of the above. Assume that our network has the following properties:
\begin{itemize}
\item Our network consists of a finite node that flows in only one direction 
(for examples, transmission of electricity, streams of water or river, and acyclic traffic).
Hence, we can regard a model of such a network as a finite poset $(P,\leq)$. Nodes are elements of $P$, and lines are ordered pairs $(p,q)$, denoted by $p \prec q$, that do not contain $r$, with $p<r<q$ in $P$.
\item It contains finitely many targets $T$ (for example, broken points, special spots, and errors or bugs) on the lines or nodes. More precisely, the targets $T$ form a discrete subset of the Hasse diagram of $P$, where we regard the diagram as a one-dimensional simplicial complex.
Furthermore, let every node have a sensor, and let it count the targets lying below the node. Here, a target $t \in T$ lying on a line or a node $p \preceq q$ is said to be below a node $r$ if $q \leq r$ in $P$. We denote this as $t \leq r$. The sensors return the counting function
\[
h : P \to \mathbb{N} \cup \{0\},
\]
given by the cardinality of targets lying below itself:
\[
h(p)=\{t \in T \mid t \leq p\}^{\sharp}.
\]
\end{itemize}

\begin{theorem}
Given the counting function $h : P \to \mathbb{N} \cup \{0\}$ for a collection of targets $T$ in a network $P$, we have
\[
T^{\sharp} = \int_{P} h d \chi.
\]
\end{theorem}
\begin{proof}
Let each target $t \in T$ lie on a line or a node $p_{t} \preceq q_{t}$. Then, the following equality holds:
\[
\int_{P} h d\chi = \int_{P} \left(\sum_{t \in T} \delta_{\vee q_{t}} \right) d \chi=  \sum_{t \in T} \chi(\vee q_{t}) = T^{\sharp}.
\]
\end{proof}

\begin{example}
The following diagram represents a counting function $h$ for targets $T$ on a network.
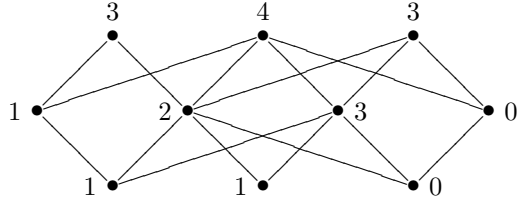
\[
\begin{xy}
\ar@{-} (10,0) *++!R{1} *{\bullet}="A_{1}";
(0,10) *++!R{1} *{\bullet}="D_{2}"
\ar@{-} (30,0) *++!R{1} *{\bullet}="B_{1}";
(20,10) *++!R{2} *{\bullet}="E_{2}"
\ar@{-} (50,0) *++!L{0} *{\bullet}="C_{1}";
(40,10) *++!L{3} *{\bullet}="F_{2}"
\ar@{-} "A_{1}";"E_{2}"
\ar@{-} "B_{1}";"F_{2}"
\ar@{-} "A_{1}";"F_{2}"
\ar@{-} "C_{1}";
(60,10) *++!L{0} *{\bullet}="G_{2}"
\ar@{-} "D_{2}";
(10,20) *++!D{3} *{\bullet}="H_{3}"
\ar@{-} "E_{2}";
(30,20) *++!D{4} *{\bullet}="I_{3}"
\ar@{-} "F_{2}";
(50,20) *++!D{3} *{\bullet}="J_{3}"
\ar@{-} "E_{2}";"H_{3}"
\ar@{-} "F_{2}";"I_{3}"
\ar@{-} "G_{2}";"J_{3}"
\ar@{-} "E_{2}";"J_{3}"
\ar@{-} "C_{1}";"E_{2}"
\ar@{-} "G_{2}";"I_{3}"
\ar@{-} "D_{2}";"I_{3}"
\end{xy}
\]
Here, we describe the underlying poset $P$ as the Hasse diagram.
Let us calculate the Euler integration of $h$ to enumerate the targets. 
The counting function is a poset map, hence $h^{-1}(\vee i)=\{ p \in P \mid h(p) \geq i\}$ is a filter of $P$ for each $i \in \mathbb{N}$.
Then, 
\[
T^{\sharp}=\int_{P}h d \chi = \int_{P}\left(\sum_{i=1}^{\infty} \delta_{h^{-1}(\vee i)} \right) d \chi = \sum_{i=1}^{\infty} \chi(h^{-1}(\vee i)). 
\]
The telescope sum above suggests that the Euler integration of $h$ can be obtained from the Euler characteristic of each level subposet $h^{-1}(\vee i)$.

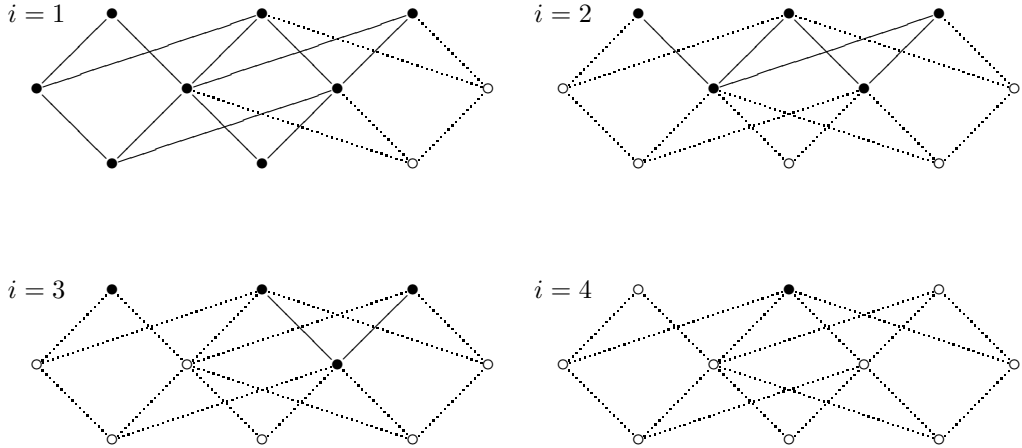
\[
\begin{xy}
(0,20) *{i=1}, (70,20) *{i=2},
\ar@{-} (10,0) *++!R{} *{\bullet}="A_{1}";
(0,10) *++!R{} *{\bullet}="D_{2}"
\ar@{-} (30,0)  *++!R{} *{\bullet}="B_{1}";
(20,10) *++!R{} *{\bullet}="E_{2}"
\ar@{.} (50,0) *++!L{} *{\circ}="C_{1}";
(40,10) *++!L{} *{\bullet}="F_{2}"
\ar@{-} "A_{1}";"E_{2}"
\ar@{-} "B_{1}";"F_{2}"
\ar@{-} "A_{1}";"F_{2}"
\ar@{.} "C_{1}";
(60,10) *++!L{} *{\circ}="G_{2}"
\ar@{-} "D_{2}";
(10,20) *++!D{} *{\bullet}="H_{3}"
\ar@{-} "E_{2}";
(30,20) *++!D{} *{\bullet}="I_{3}"
\ar@{-} "F_{2}";
(50,20) *++!D{} *{\bullet}="J_{3}"
\ar@{-} "E_{2}";"H_{3}"
\ar@{-} "F_{2}";"I_{3}"
\ar@{.} "G_{2}";"J_{3}"
\ar@{-} "E_{2}";"J_{3}"
\ar@{.} "C_{1}";"E_{2}"
\ar@{.} "G_{2}";"I_{3}"
\ar@{-} "D_{2}";"I_{3}"
\ar@{.} (80,0) *++!R{} *{\circ}="A_{1}";
(70,10) *++!R{} *{\circ}="D_{2}"
\ar@{.} (100,0)  *++!R{} *{\circ}="B_{1}";
(90,10) *++!R{} *{\bullet}="E_{2}"
\ar@{.} (120,0) *++!L{} *{\circ}="C_{1}";
(110,10) *++!L{} *{\bullet}="F_{2}"
\ar@{.} "A_{1}";"E_{2}"
\ar@{.} "B_{1}";"F_{2}"
\ar@{.} "A_{1}";"F_{2}"
\ar@{.} "C_{1}";
(130,10) *++!L{} *{\circ}="G_{2}"
\ar@{.} "D_{2}";
(80,20) *++!D{} *{\bullet}="H_{3}"
\ar@{-} "E_{2}";
(100,20) *++!D{} *{\bullet}="I_{3}"
\ar@{-} "F_{2}";
(120,20) *++!D{} *{\bullet}="J_{3}"
\ar@{-} "E_{2}";"H_{3}"
\ar@{-} "F_{2}";"I_{3}"
\ar@{.} "G_{2}";"J_{3}"
\ar@{-} "E_{2}";"J_{3}"
\ar@{.} "C_{1}";"E_{2}"
\ar@{.} "G_{2}";"I_{3}"
\ar@{.} "D_{2}";"I_{3}"
\end{xy}
\]

\[
\begin{xy}
(0,20) *{i=3}, (70,20) *{i=4},
\ar@{.} (10,0) *++!R{} *{\circ}="A_{1}";
(0,10) *++!R{} *{\circ}="D_{2}"
\ar@{.} (30,0)  *++!R{} *{\circ}="B_{1}";
(20,10) *++!R{} *{\circ}="E_{2}"
\ar@{.} (50,0) *++!L{} *{\circ}="C_{1}";
(40,10) *++!L{} *{\bullet}="F_{2}"
\ar@{.} "A_{1}";"E_{2}"
\ar@{.} "B_{1}";"F_{2}"
\ar@{.} "A_{1}";"F_{2}"
\ar@{.} "C_{1}";
(60,10) *++!L{} *{\circ}="G_{2}"
\ar@{.} "D_{2}";
(10,20) *++!D{} *{\bullet}="H_{3}"
\ar@{.} "E_{2}";
(30,20) *++!D{} *{\bullet}="I_{3}"
\ar@{-} "F_{2}";
(50,20) *++!D{} *{\bullet}="J_{3}"
\ar@{.} "E_{2}";"H_{3}"
\ar@{-} "F_{2}";"I_{3}"
\ar@{.} "G_{2}";"J_{3}"
\ar@{.} "E_{2}";"J_{3}"
\ar@{.} "C_{1}";"E_{2}"
\ar@{.} "G_{2}";"I_{3}"
\ar@{.} "D_{2}";"I_{3}"
\ar@{.} (80,0) *++!R{} *{\circ}="A_{1}";
(70,10) *++!R{} *{\circ}="D_{2}"
\ar@{.} (100,0)  *++!R{} *{\circ}="B_{1}";
(90,10) *++!R{} *{\circ}="E_{2}"
\ar@{.} (120,0) *++!L{} *{\circ}="C_{1}";
(110,10) *++!L{} *{\circ}="F_{2}"
\ar@{.} "A_{1}";"E_{2}"
\ar@{.} "B_{1}";"F_{2}"
\ar@{.} "A_{1}";"F_{2}"
\ar@{.} "C_{1}";
(130,10) *++!L{} *{\circ}="G_{2}"
\ar@{.} "D_{2}";
(80,20) *++!D{} *{\circ}="H_{3}"
\ar@{.} "E_{2}";
(100,20) *++!D{} *{\bullet}="I_{3}"
\ar@{.} "F_{2}";
(120,20) *++!D{} *{\circ}="J_{3}"
\ar@{.} "E_{2}";"H_{3}"
\ar@{.} "F_{2}";"I_{3}"
\ar@{.} "G_{2}";"J_{3}"
\ar@{.} "E_{2}";"J_{3}"
\ar@{.} "C_{1}";"E_{2}"
\ar@{.} "G_{2}";"I_{3}"
\ar@{.} "D_{2}";"I_{3}"
\end{xy}
\]
The classifying space $B(h^{-1}(\vee 1))$ is homotopy equivalent to a sphere $S^{2}$, 
and $B(h^{-1}(\vee 2))$ is homotopy equivalent to a circle $S^{1}$. Proposition \ref{acyclic} leads to the following result:
\[
T^{\sharp}=\chi(h^{-1}(\vee 1)) + \chi(h^{-1}(\vee 2)) + \chi(h^{-1}(\vee 3))+\chi(h^{-1}(\vee 4))=2+0+2+1=5.
\]

Indeed, this counting function was given by the following five targets described as ``$+$".
\[
\begin{xy}
\ar@{-} (10,0) *+{+} *++!R{1} *{\bullet}="A_{1}";
(0,10) *++!R{1} *{\bullet}="D_{2}"
\ar@{-} (30,0) *+{+} *++!R{1} *{\bullet}="B_{1}";
(20,10) *++!R{2} *{\bullet}="E_{2}"
\ar@{-}|{+} (50,0) *++!L{0} *{\bullet}="C_{1}";
(40,10) *++!L{3} *{\bullet}="F_{2}"
\ar@{-} "A_{1}";"E_{2}"
\ar@{-} "B_{1}";"F_{2}"
\ar@{-} "A_{1}";"F_{2}"
\ar@{-} "C_{1}";
(60,10) *++!L{0} *{\bullet}="G_{2}"
\ar@{-}|{+} "D_{2}";
(10,20) *++!D{3} *{\bullet}="H_{3}"
\ar@{-}|{+} "E_{2}";
(30,20) *++!D{4} *{\bullet}="I_{3}"
\ar@{-} "F_{2}";
(50,20) *++!D{3} *{\bullet}="J_{3}"
\ar@{-} "E_{2}";"H_{3}"
\ar@{-} "F_{2}";"I_{3}"
\ar@{-} "G_{2}";"J_{3}"
\ar@{-} "E_{2}";"J_{3}"
\ar@{-} "C_{1}";"E_{2}"
\ar@{-} "G_{2}";"I_{3}"
\ar@{-} "D_{2}";"I_{3}"
\end{xy}
\]

\end{example}


Institute of Social Sciences, School of Humanities and Social Sciences, Academic Assembly, Shinshu University, Japan.

\textit{E-mail address}: tanaka@shinshu-u.ac.jp

\end{document}